\documentclass[10pt]{article}

\usepackage{hyperref}
\usepackage{xcolor}
\usepackage{amssymb}
\usepackage{amsfonts}
\usepackage{amsthm}
\usepackage{amsmath}
\usepackage{float}
\usepackage{bigints}
\usepackage{fullpage}
\usepackage{mathtools}
\usepackage[utf8]{inputenc} 
\usepackage{cancel}
\usepackage{verbatim}

\newtheorem{theorem}{Theorem}

\newtheorem{lemma}[theorem]{Lemma}

\newtheorem{proposition}[theorem]{Proposition}

\newtheorem{definition}{Definition}

\newtheorem{remark}[theorem]{Remark}

\newcommand{\inter}{\rm int}

\newcommand{\x}{\mathbf{x}}
\newcommand{\p}{\mathbf{p}}
\newcommand{\q}{\mathbf{q}}
\newcommand{\y}{\mathbf{y}}

\newcommand{\B}{\mathbf{B}}

\newcommand{\R}{\mathbf{R}_{+}}

\newcommand{\Ee}{\mathbb{E}}
\newcommand{\Hh}{\mathbb{H}}
\newcommand{\Ss}{\mathbb{S}}
\newcommand{\Mm}{\mathbb{M}}
\newcommand{\Zz}{\mathbb{Z}}
\newcommand{\Rr}{\mathbb{R}}

\newcommand{\dist}{{\rm dist}}

\title{From $r$-dual sets to uniform contractions
\footnote{\hspace*{.35cm} 2010 Mathematics Subject Classification: 52A20, 52A22.}}

\author{K\'{a}roly Bezdek\thanks{Partially supported by a Natural Sciences and 
Engineering Research Council of Canada Discovery Grant.}
}

\date{}

\begin{document}

\maketitle

\begin{abstract}
Let $\Mm^d$ denote the $d$-dimensional Euclidean, hyperbolic, or spherical space. The $r$-dual set of given set in $\Mm^d$ is the intersection of balls of radii $r$ centered at the points of the given set. In this paper we prove that for any set of given volume in $\Mm^d$ the volume of the $r$-dual set becomes maximal if the set is a ball.
As an application we prove the following. The Kneser--Poulsen Conjecture states that if the centers of a family of $N$ congruent balls in Euclidean $d$-space is contracted, 
then the volume of the intersection does not decrease. A uniform contraction is a contraction where all the pairwise distances in the first set of centers are larger than all the pairwise distances in the second set of centers. We prove the Kneser--Poulsen conjecture for uniform contractions (with $N$ sufficiently large) in $\Mm^d$.
\end{abstract}

\section{Introduction}\label{sec:intro}
Let $\Mm^d$, $d>1$ denote the $d$-dimensional Euclidean, hyperbolic, or spherical space, i.e., one of the simply connected complete Riemannian manifolds of constant sectional curvature. Since simply connected complete space forms, the sectional curvature of which have the same sign are similar, we may assume without loss of generality that the sectional curvature $\kappa$ of $\Mm^d$ is $0, -1$, or $1$. Let $\R$ denote the set of positive real numbers for $\kappa\leq 0$ and the half-open interval $(0, \frac{\pi}{2}]$ for $\kappa =1$. Let $\dist_{\Mm^d}(\x,\y)$ stand for the geodesic distance between the points $\x\in{\Mm^d}$ and $\y\in{\Mm^d}$. Furthermore, let $\B_{\Mm^d}[\x, r]$ denote the closed $d$-dimensional ball with center $\x\in\Mm^d$ and radius $r\in\R$ in $\Mm^d$, i.e., let $\B_{\Mm^d}[\x, r]:=\{\y\in\Mm^d\ |\dist_{\Mm^d}(\x,\y)\leq r\}$. Now, we are ready to introduce the central notion of this paper.

\begin{definition}\label{r-dual-body}
For a set $X\subseteq\Mm^d$, $d>1$ and $r\in\R$ let the $r$-dual set $X^r$ of $X$ be defined by $X^r:=\bigcap_{\x\in X}\B_{\Mm^d}[\x, r]$. If the interior ${\inter} (X^r)\neq\emptyset$, then we call $X^r$ the $r$-dual body of $X$.
\end{definition}

We note that either $X^r=\emptyset$, or $X^r$ is a point in $\Mm^d$, or ${\inter} (X^r)\neq\emptyset$. Perhaps not surprisingly, $r$-dual sets of $\Ee^d$ have already been investigated in a number of papers however, under various names such as "\"uberkonvexe Menge" (\cite{Ma}), "$r$-convex domain" (\cite{Fe}), "spindle convex set" (\cite{BLNP}, \cite{KMP}), "ball convex set" (\cite{LNT}), and "hyperconvex set" ({\cite{FKV}). $r$-dual sets satisfy some basic identities such as $$\left((X^r)^r\right))^r=X^r \ {\rm and} \ (X \cup Y)^r=X^r\cap Y^r,$$ which hold for any $X\subseteq\Mm^d$ and $Y\subseteq\Mm^d$. Clearly, also monotonicity holds namely, $X\subseteq Y\subseteq\Mm^d$ implies $Y^r\subseteq X^r$. Thus, there is a good deal of similarity between $r$-dual sets and polar sets (resp., spherical polar sets) in $\Ee^d$ (resp., $\Ss^d$). In this paper we explore further this similarity by investigating a volumetric relation between $X^r$ and $X$ in $\Mm^d$. For this reason let $V_{\Mm^d}(\cdot)$ denote the Lebesgue measure in $\Mm^d$, to which we are going to refer as volume in $\Mm^d$. Now, recall the recent theorem of Gao, Hug, and Schneider \cite{GHSch} stating that for any convex body of given volume in $\Ss^d$ the volume of the spherical polar body becomes maximal if the convex body is a ball. We prove the following extension of their theorem.

\begin{theorem}\label{main}
Let $A\subseteq\Mm^d$, $d>1$ be a compact set of volume $V_{\Mm^d}(A)>0$ and $r\in\R$. If $B\subseteq\Mm^d$ is a ball with $V_{\Mm^d}(A)=V_{\Mm^d}(B)$, then $V_{\Mm^d}(A^r)\leq V_{\Mm^d}(B^r)$.
\end{theorem}

Note that the Gao--Hug--Schneider theorem is a special case of Theorem~\ref{main} namely, when $\Mm^d=\Ss^d$ and $r=\frac{\pi}{2}$. As this theorem of \cite{GHSch} is often called a spherical counterpart of the Blaschke--Santal\'o inequality, one may refer to Theorem~\ref{main} as a Blaschke--Santal\'o-type inequality for $r$-duality in $\Mm^d$.

From our point view, the importance of Theorem~\ref{main}  lies in the following application. For stating it in a proper way we recall the following notion from \cite{BeNa}.

\begin{definition}
We say that the (labeled) point set $\{\q_1,\dots ,\q_N\}\subset\Mm^d$ is a uniform contraction of the (labeled) point set $\{\p_1,\dots ,\p_N\}\subset\Mm^d$ with separating value $\lambda>0$ in $\Mm^d$, $d>1$ if 
$$\dist_{\Mm^d}(\q_i,\q_j)\leq\lambda\leq\dist_{\Mm^d}(\p_i,\p_j)$$
holds for all $1\leq i<j\leq N$.
\end{definition}

Now, recall the following recent theorem of the author and Nasz\'odi \cite{BeNa} : Let $d\in\Zz$ and $\delta, \lambda\in\Rr$ be given such that $d>1$ and $0<\lambda\leq\sqrt{2}\delta$. If $Q:=\{\q_1,\dots ,\q_N\}\subset\Ee^d$ is a uniform contraction of $P:=\{\p_1,\dots ,\p_N\}\subset\Ee^d$ with separating value $\lambda$ in $\Ee^d$ and $N\geq(1+\sqrt{2})^d$, then $V_{\Ee^d}(P^{\delta})\leq V_{\Ee^d}(Q^{\delta})$. As it is explained in \cite{BeNa}, this proves the Kneser--Poulsen conjecture for uniform contractions. For the sake of completeness we mention here that according to the Kneser--Poulsen conjecture if a finite set of balls in $\Ee^d$ is rearranged so that the distance between each pair of centers does not increase, then the volume of the intersection does not decrease. This is proved for $d=2$ in \cite{BeCo} and it remains open for $d>2$. For more details on the Kneser--Poulsen conjecture we refer the interested reader to Chapter 3 in \cite{Be}. In this paper, we give a rather short and elementary proof of the above mentioned theorem of the author and Nasz\'odi (replacing the Brunn--Minkowski inequality in \cite{BeNa} by Theorem~\ref{main}) and perhaps, more importantly we extend it to hyperbolic as well as spherical spaces as follows.


\begin{theorem}\label{appl}
\item (i) Let $d\in\Zz$ and $\delta, \lambda\in\Rr$ be given such that $d>1$ and $0<\lambda\leq\sqrt{2}\delta$. If $Q:=\{\q_1,\dots ,\q_N\}\subset\Ee^d$ is a uniform contraction of $P:=\{\p_1,\dots ,\p_N\}\subset\Ee^d$ with separating value $\lambda$ in $\Ee^d$ and $N\geq(1+\sqrt{2})^d$, then $V_{\Ee^d}(P^{\delta})< V_{\Ee^d}(Q^{\delta})$.
\item (ii) Let $d\in\Zz$ and $\delta, \lambda\in\Rr$ be given such that $d>1, 0<\delta<\frac{\pi}{2}$, and $0<\lambda< \min\left\{\frac{2\sqrt{2}}{\pi}\delta,\pi-2\delta\right\}$. If $Q:=\{\q_1,\dots ,\q_N\}\subset\Ss^d$ is a uniform contraction of $P:=\{\p_1,\dots ,\p_N\}\subset\Ss^d$ with separating value $\lambda$ in $\Ss^d$ and $N\geq 2ed{\pi}^{d-1}\left(\frac{1}{2}+\frac{\pi}{2\sqrt{2}}\right)^d$, then $V_{\Ss^d}(P^{\delta})< V_{\Ss^d}(Q^{\delta})$.
\item (iii) Let $d, k\in\Zz$ and $\delta, \lambda\in\Rr$ be given such $d>1, k>0$ and $0<\frac{\sinh k}{\sqrt{2}k}\lambda\leq\delta<k$. If $Q:=\{\q_1,\dots ,\q_N\}\subset\Hh^d$ is a uniform contraction of $P:=\{\p_1,\dots ,\p_N\}\subset\Hh^d$ with separating value $\lambda$ in $\Hh^d$ and $N\geq (\frac{\sinh  2k}{2k})^{d-1}(\frac{\sqrt{2}\sinh k}{k}+1)^d$, then $V_{\Hh^d}(P^{\delta})< V_{\Hh^d}(Q^{\delta})$.
\end{theorem}

In the rest of the paper we prove the theorems stated. 

\section{Proof of Theorem~\ref{main}}

We adapt the two-point symmetrization method of the proof of the Gao-Hug-Schneider theorem from \cite{GHSch}. For this we need to recall the definition of two-point  symmetrization, which is also known under the names "two-point rearrangement", "compression", or "polarization". (For more details on two-point symmetrization we refer the interested reader to  the relevant section in \cite{GHSch} and the references mentioned there.)

\begin{definition}
Let $H$ be a hyperplane in $\Mm^d$ with an orientation, which determines $H^+$ and $H^-$ the two closed halfspaces bounded by $H$ in $\Mm^d$, $d>1$. Let $\sigma_H$ denote the reflection about $H$ in $\Mm^d$. If $K\subseteq\Mm^d$, then the two-point symmetrization $\tau_H$ with respect to $H$ transforms $K$ into the set
$$\tau_H K:=(K\cap\sigma_H K)\cup\left((K\cup\sigma_H K)\cap H^+\right).$$
If $K_H:=K\cap\sigma_H K$ stands for the $H$-symmetric core of $K$, then we call 
\begin{equation}\label{disjoint}
\tau_H K=K_H\cup\left((K\cap H^+)\setminus K_H\right)\cup\sigma_H \left((K\cap H^-)\setminus K_H\right)
\end{equation}
the canonical decomposition of $\tau_H K$.
\end{definition}


\begin{remark}\label{basic1}
The canonical decomposition of $\tau_H K$ is a disjoint decomposition of $\tau_H K$, which easily implies that two-point symmetrization preserves volume.
\end{remark}

\begin{definition}
Let $K\subset\Mm^d$, $d>1$ and $r\in\R$. Then the $r$-convex hull ${\rm conv}_rK$ of $K$ is defined by $${\rm conv}_rK:=\bigcap\{ \B_{\Mm^d}[\x, r]\ |\ K\subseteq \B_{\Mm^d}[\x, r]\}.$$
Moreover, let the $r$-convex hull of $\Mm^d$ be $\Mm^d$. Furthermore, we say that  $K\subseteq\Mm^d$ is an $r$-convex set if $K={\rm conv}_rK$.
\end{definition}

\begin{lemma}\label{basic2}
If $K\subseteq\Mm^d$, $d>1$ and $r\in\R$, then 
\begin{equation}\label{b-2}
K^r= ({\rm conv}_rK)^r.
\end{equation}
\end{lemma}

\begin{proof}
Clearly, $K\subseteq {\rm conv}_rK$ and therefore $({\rm conv}_rK)^r\subseteq K^r$. On the other hand, we show that $K^r\subseteq ({\rm conv}_rK)^r$. As this holds trivially for $K^r=\emptyset$, we may assume that $K^r\neq\emptyset$. So let $\y\in K^r$. Then it is clear that $K\subseteq \B_{\Mm^d}[\y, r]$ and so, ${\rm conv}_rK\subseteq  \B_{\Mm^d}[\y, r]$ implying that $\y\in ({\rm conv}_rK)^r$. Thus, (\ref{b-2}) follows.
\end{proof}

The core part of our proof of Theorem~\ref{main} is

\begin{lemma}\label{core}
If $K\subseteq\Mm^d$, $d>1$ and $r\in\R$, then $$\tau_H (K^r)\subseteq \left({\rm conv}_r(\tau_H K)\right)^r.$$
\end{lemma}

\begin{proof} Lemma~\ref{basic2} implies that $\left({\rm conv}_r(\tau_H K)\right)^r=\left(\tau_H K\right)^r$ and so, it is sufficient to prove that $\tau_H (K^r)\subseteq \left(\tau_H K\right)^r$. For this we need to show that if $\x\in \tau_H (K^r)$, then $\x\in \left(\tau_H K\right)^r$, i.e., 
\begin{equation}\label{goal}
\tau_H K\subseteq  \B_{\Mm^d}[\x, r].
\end{equation}

Remark~\ref{basic1} implies that
$$\tau_H (K^r)=(K^r)_H\cup\left((K^r\cap H^+)\setminus (K^r)_H\right)\cup\sigma_H \left((K^r\cap H^-)\setminus (K^r)_H\right)$$
is a disjoint decomposition of $\tau_H (K^r)$ with $(K^r)_H=K^r\cap\sigma_H(K^r)$. Thus, either $\x\in (K^r)_H$ (Case 1), or $\x\in (K^r\cap H^+)\setminus (K^r)_H$ (Case 2), or
$\x\in \sigma_H \left((K^r\cap H^-)\setminus (K^r)_H\right)$ (Case 3). In all three cases we use (\ref{disjoint}) for the proof of (\ref{goal}).

\underline{{\it Case 1:}} As $(K^r)_H=K^r\cap\sigma_H(K^r)$ therefore $\x,\sigma_H\x\in (K^r)_H$. As $\x\in (K^r)_H\subseteq K^r$ threrefore $K_H\cup\left((K\cap H^+)\setminus K_H\right)\subseteq K\subseteq  \B_{\Mm^d}[\x, r]$. On the other hand, as $\sigma_H\x\in (K^r)_H\subseteq K^r$ therefore $(K\cap H^-)\setminus K_H \subseteq K\subseteq  \B_{\Mm^d}[\sigma_H\x, r]$ and so, $\sigma_H\left((K\cap H^-)\setminus K_H\right) \subseteq  \B_{\Mm^d}[\x, r]$, finishing the proof of (\ref{goal}).

\underline{{\it Case 2:}} As $\x\in (K^r\cap H^+)\setminus (K^r)_H\subseteq K^r$ therefore $K_H\cup\left((K\cap H^+)\setminus K_H\right)\subseteq K\subseteq  \B_{\Mm^d}[\x, r]$. So, we are left to show that 
\begin{equation}\label{sub-goal}
\sigma_H \left((K\cap H^-)\setminus K_H\right)\subseteq  \B_{\Mm^d}[\x, r].
\end{equation} 
On the one hand, $\x\in (K^r\cap H^+)\setminus (K^r)_H\subseteq K^r$ implies that $(K\cap H^-)\setminus K_H \subseteq K\subseteq  \B_{\Mm^d}[\x, r]$. On the other hand, for any $\y\in (K\cap H^-)\setminus K_H$ we have $\sigma_H \y\in\sigma_H \left((K\cap H^-)\setminus K_H\right)$. As $\x, \sigma_H\y\in H^+$ and $\y\in H^-$ therefore ${\rm dist}_{\Mm^d}(\sigma_H\y,\x)\leq {\rm dist}_{\Mm^d}(\y,\x)\leq r$. Thus, (\ref{sub-goal}) follows.

\underline{{\it Case 3:}} It follows from the assumption that $\sigma_H\x\in (K^r\cap H^-)\setminus (K^r)_H\subseteq K^r$ and therefore $(K\cap H^-)\setminus K_H \subseteq K\subseteq  \B_{\Mm^d}[\sigma_H\x, r]$ implying that $\sigma_H \left((K\cap H^-)\setminus K_H\right)\subseteq  \B_{\Mm^d}[\x, r]$. So, we are left to show that
\begin{equation} \label{final-goal}
K_H\cup\left((K\cap H^+)\setminus K_H\right)\subseteq  \B_{\Mm^d}[\x, r].
\end{equation}
As $\sigma_H\x\in (K^r\cap H^-)\setminus (K^r)_H\subseteq K^r$ therefore $K_H\cup\left((K\cap H^+)\setminus K_H\right)\subseteq K  \subseteq  \B_{\Mm^d}[\sigma_H\x, r]$. Moreover, as $\sigma_H\x\in H^-$ and $\x\in H^+$ therefore for all $\y\in (K\cap H^+)\setminus K_H\subseteq H^+$ (resp., $\y\in K_H\cap H^+\subseteq H^+$) we have
${\rm dist}_{\Mm^d}(\x, \y)\leq {\rm dist}_{\Mm^d}(\sigma_H\x,\y)\leq r$ implying that $\left(K_H\cap H^+\right)\cup\left((K\cap H^+)\setminus K_H\right)\subseteq  \B_{\Mm^d}[\x, r]$. Finally, for any $\y\in K_H\cap H^-$ we have $\sigma_H\y\in K_H\cap H^+\subseteq K_H$ with ${\rm dist}_{\Mm^d}(\x,\y)={\rm dist}_{\Mm^d}(\sigma_H\x,\sigma_H\y)\leq r$
implying that $K_H\cap H^-\subseteq  \B_{\Mm^d}[\x, r]$. This completes the proof of (\ref{final-goal}).
\end{proof}

Now, we are ready to prove Theorem~\ref{main}. To avoid any trivial case we may assume that $V_{\Mm^d}(A^r)>0$ for $A\subseteq\Mm^d$ with $a:=V_{\Mm^d}(A)>0$. In fact, our goal is to maximize the volume $V_{\Mm^d}(A^r)$ for compact sets $A\subseteq\Mm^d$ of given volume $V_{\Mm^d}(A)=a>0$ and for given $d>1$ and $r\in\R$.  As according to Lemma~\ref{basic2} we have $A^r= ({\rm conv}_rA)^r$ with $A\subseteq {\rm conv}_rA$, it follows
from the monotonicity of $V_{\Mm^d}\left((\cdot)^r\right)$ in a straightforward way that for the proof of Theorem~\ref{main} it is sufficient to maximize the volume $V_{\Mm^d}(A^r)$ for $r$-convex sets $A\subseteq\Mm^d$ of given volume $V_{\Mm^d}(A)=a$ with given $d$ and $r$. Next, consider the extremal family $\mathcal E_{a,r,d}$ of $r$-convex sets $A\subseteq\Mm^d$ with $V_{\Mm^d}(A)=a$ and maximal $V_{\Mm^d}(A^r)$ for given $a$, $d$ and $r$. By standard arguments, $\mathcal E_{a,r,d}\neq\emptyset$. 

\begin{lemma}\label{closed}
The extremal family $\mathcal E_{a,r,d}$ is closed under two-point symmetrization.  
\end{lemma}

\begin{proof}
Let $A\in \mathcal E_{a,r,d}$ be an arbitrary extremal set and consider $\tau_H A$ for an arbitrary hyperplane $H$ in $\Mm^d$. Lemmas~\ref{basic2} and~\ref{core} imply that 
$\tau_H (A^r)\subseteq \left({\rm conv}_r(\tau_H A)\right)^r=(\tau_H A)^r$ and therefore 
\begin{equation}\label{ext1}
V_{\Mm^d}(A^r)=V_{\Mm^d}\left(\tau_H (A^r)\right)\leq V_{\Mm^d}\left(\left({\rm conv}_r(\tau_H A)\right)^r\right)=V_{\Mm^d}((\tau_H A)^r).
\end{equation}
Here $\tau_H A\subseteq {\rm conv}_r(\tau_H A)$ implying that 
\begin{equation}\label{ext2}
a=V_{\Mm^d}(A)=V_{\Mm^d}(\tau_H A)\leq V_{\Mm^d}\left({\rm conv}_r(\tau_H A)\right).
\end{equation}

We are left to show that $\tau_H A\in \mathcal E_{a,r,d}$. Based on (\ref{ext1}) and (\ref{ext2}) we need to prove only that $\tau_H A$ is $r$-convex, i.e., $\tau_H A={\rm conv}_r(\tau_H A)$. We prove this in indirect way, i.e., assume that $\tau_H A\neq {\rm conv}_r(\tau_H A)$. As $\tau_H A\subseteq {\rm conv}_r(\tau_H A)$, this means that $\tau_H A\subset {\rm conv}_r(\tau_H A)$. Then there exists an $r$-convex set $A'\subset {\rm conv}_r(\tau_H A)$ with $V_{\Mm^d}(A')=a$. Thus, $\left({\rm conv}_r(\tau_H A)\right)^r\subset(A')^r$ implying that $V_{\Mm^d}\left(\left({\rm conv}_r(\tau_H A)\right)^r\right)< V_{\Mm^d}\left( (A')^r\right)$, a contradiction via (\ref{ext1}).
\end{proof}

We finish the proof of Theorem~\ref{main} by adapting an argument from \cite{GHSch}. Namely, we are going to show that $B\in\mathcal E_{a,r,d}$, where $B\subseteq\Mm^d$ is a ball with $a=V_{\Mm^d}(A)=V_{\Mm^d}(B)$. By a standard argument there exists an $r$-convex set $C\in \mathcal E_{a,r,d}$ for which $V_{\Mm^d}(B\cap C)$ is maximal.
Suppose that $B\neq C$. As $a=V_{\Mm^d}(B)=V_{\Mm^d}(C)$ therefore there are congruent balls $C_1\subseteq C\setminus B$ and $C_2\subseteq B\setminus C$. Let $H$ be the hyperplane in $\Mm^d$ with an orientation, which determines $H^+$ and $H^-$ the two closed halfspaces bounded by $H$ in $\Mm^d$, $d>1$ such that $\sigma_H B_1=B_2$ with $B_1\subset H^-$. Clearly, $V_{\Mm^d}(B\cap\tau_H C)>V_{\Mm^d}(B\cap C)$ moreover, Lemma~\ref{closed} implies that $\tau_H C\in \mathcal E_{a,r,d}$, a contradiction. Thus, $B=C\in\mathcal E_{a,r,d}$, finishing the proof of Theorem~\ref{main}.

\section{Proof of Theorem~\ref{appl}}

Following \cite{BeNa}, our proof is based on estimates of the following functionals.

\begin{definition} Let 
\begin{equation}\label{d-1}
f_{\Mm^d}(N,\lambda,\delta):=\min\{V_{\Mm^d}(Q^{\delta})\ |\ Q:=\{\q_1,\dots ,\q_N\}\subset\Mm^d,\  \dist_{\Mm^d}(\q_i,\q_j)\leq\lambda\ {\rm for\ all} \ 1\leq i<j\leq N \}
\end{equation}
and
\begin{equation}\label{d-2}
g_{\Mm^d}(N,\lambda,\delta):=\max\{V_{\Mm^d}(P^{\delta})\ |\ P:=\{\p_1,\dots ,\p_N\}\subset\Mm^d,\  \lambda\leq\dist_{\Mm^d}(\p_i,\p_j)\ {\rm for\ all} \ 1\leq i<j\leq N \}
\end{equation}
\end{definition}

(We note that in this paper the maximum of the empty set is zero.) We need also

\begin{definition}
The circumradius ${\rm cr} X$ of the set $X\subseteq\Mm^d$, $d>1$  is defined by
$$
{\rm cr}X:=\inf\{r\ |\ X\subseteq \B_{\Mm^d}[\x, r]\}.
$$
\end{definition}

\subsection{Proof of (i) in Theorem~\ref{appl}}

First, we give a lower bound for (\ref{d-1}). Jung's theorem (\cite{De}) implies in a straightforward way that ${\rm cr} Q\leq\sqrt{\frac{2d}{d+1}}\frac{\lambda}{2}<\frac{1}{\sqrt{2}}\lambda$ and so, $\B_{\Ee^d}\left[\x, \delta-\frac{1}{\sqrt{2}}\lambda\right] \subset Q^{\delta}$ for some $\x\in\Ee^d$. (We note that by assumption $\delta-\frac{1}{\sqrt{2}}\lambda\geq 0$.) As a result we get that 
\begin{equation}\label{1-E}
f_{\Ee^d}(N,\lambda,\delta)>V_{\Ee^d}\left(\B_{\Ee^d}\left[\x, \delta-\frac{1}{\sqrt{2}}\lambda\right]\right).
\end{equation}
Second, we give an upper bound for (\ref{d-2}). It follows in a straightforward way that 
\begin{equation}\label{2-E}
P^{\delta}=\left(\bigcup_{i=1}^{N}\B_{\Ee^d}\left[\p_i,\frac{\lambda}{2}\right]\right)^{\delta+\frac{\lambda}{2}}, 
\end{equation}
where the balls $\B_{\Ee^d}[\p_1,\frac{\lambda}{2}],\dots ,\B_{\Ee^d}[\p_N,\frac{\lambda}{2}]$ are pairwise non-overlapping in $\Ee^d$. Thus, 
\begin{equation}\label{3-E}
V_{\Ee^d}\left(\bigcup_{i=1}^{N}\B_{\Ee^d}\left[\p_i,\frac{\lambda}{2}\right]\right)=N V_{\Ee^d}\left( \B_{\Ee^d}\left[\p_1,\frac{\lambda}{2}\right]\right).
\end{equation}
Let $\mu>0$ be chosen such that $N V_{\Ee^d}\left( \B_{\Ee^d}\left[\p_1,\frac{\lambda}{2}\right]\right)=V_{\Ee^d}\left( \B_{\Ee^d}\left[\p_1,\mu\right]\right)$. Clearly,
\begin{equation}\label{4-E}
\mu=\frac{1}{2}N^{\frac{1}{d}}\lambda
\end{equation}
Now Theorem~\ref{main}, (\ref{2-E}), (\ref{3-E}), and (\ref{4-E}) imply in a straightforward way that
\begin{equation}\label{5-E}
V_{\Ee^d}\left(P^{\delta}\right)=V_{\Ee^d}\left(\left(\bigcup_{i=1}^{N}\B_{\Ee^d}\left[\p_i,\frac{\lambda}{2}\right]\right)^{\delta+\frac{\lambda}{2}}\right)\leq
V_{\Ee^d}\left(\left(\B_{\Ee^d}\left[\p_1, \frac{1}{2}N^{\frac{1}{d}}\lambda\right]\right)^{\delta+\frac{\lambda}{2}}\right)
\end{equation}
Clearly, $\left(\B_{\Ee^d}\left[\p_1, \frac{1}{2}N^{\frac{1}{d}}\lambda\right]\right)^{\delta+\frac{\lambda}{2}}=\B_{\Ee^d}\left[\p_1,\delta-\frac{N^{\frac{1}{d}}-1}{2}\lambda\right]$ with the convention that if $\delta-\frac{N^{\frac{1}{d}}-1}{2}\lambda<0$, then $\B_{\Ee^d}\left[\p_1,\delta-\frac{N^{\frac{1}{d}}-1}{2}\lambda\right]=\emptyset$. Hence (\ref{5-E}) yields
\begin{equation}\label{6-E}
g_{\Ee^d}(N,\lambda,\delta)\leq V_{\Ee^d}\left(\B_{\Ee^d}\left[\p_1,\delta-\frac{N^{\frac{1}{d}}-1}{2}\lambda\right]\right)
\end{equation}
(with the convention that $V_{\Ee^d}(\emptyset)=0$). Finally, as $N\geq(1+\sqrt{2})^d$ therefore $\frac{N^{\frac{1}{d}}-1}{2}\lambda\geq \frac{1}{\sqrt{2}}\lambda$ and so, (\ref{1-E}) and (\ref{6-E}) yield $g_{\Ee^d}(N,\lambda,\delta)<
f_{\Ee^d}(N,\lambda,\delta)$, finishing the proof of (i) in Theorem~\ref{appl}.

\subsection{Proof of (ii) in Theorem~\ref{appl}}

First, we lower bound (\ref{d-1}). Let $R:={\rm cr}Q$. Then Jung's theorem (\cite{De}) yields $\sin R\leq\sqrt{\frac{2d}{d+1}}\sin\frac{\lambda}{2}$. By assumption $0<\lambda<\frac{\pi}{2}$ and so, 
$$
0\leq\frac{2}{\pi}R<\sin R\leq\sqrt{\frac{2d}{d+1}}\sin\frac{\lambda}{2}<\sqrt{\frac{2d}{d+1}}\frac{\lambda}{2}<\frac{1}{\sqrt{2}}\lambda
$$ 
implying that $0\leq R<\frac{\pi}{2\sqrt{2}}\lambda$. Thus, $\B_{\Ss^d}\left[\x, \delta-\frac{\pi}{2\sqrt{2}}\lambda\right] \subset Q^{\delta}$ for some $\x\in\Ss^d$. (We note that by assumption $\delta-\frac{\pi}{2\sqrt{2}}\lambda> 0$.) As a result we get that
\begin{equation}\label{1-S}
f_{\Ss^d}(N,\lambda,\delta)>V_{\Ss^d}\left(\B_{\Ss^d}\left[\x, \delta-\frac{\pi}{2\sqrt{2}}\lambda\right]\right).
\end{equation}
Second, we upper bound (\ref{d-2}). It follows in a straightforward way that 
\begin{equation}\label{2-S}
P^{\delta}=\left(\bigcup_{i=1}^{N}\B_{\Ss^d}\left[\p_i,\frac{\lambda}{2}\right]\right)^{\delta+\frac{\lambda}{2}}, 
\end{equation}
where the balls $\B_{\Ss^d}[\p_1,\frac{\lambda}{2}],\dots ,\B_{\Ss^d}[\p_N,\frac{\lambda}{2}]$ are pairwise non-overlapping in $\Ss^d$. Thus, 
\begin{equation}\label{3-S}
V_{\Ss^d}\left(\bigcup_{i=1}^{N}\B_{\Ss^d}\left[\p_i,\frac{\lambda}{2}\right]\right)=N V_{\Ss^d}\left( \B_{\Ss^d}\left[\p_1,\frac{\lambda}{2}\right]\right).
\end{equation}
Let $\mu>0$ be chosen such that 
\begin{equation}\label{4-S}
N V_{\Ss^d}\left( \B_{\Ss^d}\left[\p_1,\frac{\lambda}{2}\right]\right)=V_{\Ss^d}\left( \B_{\Ss^d}\left[\p_1,\mu\right]\right).
\end{equation}

\begin{proposition}\label{S-mu} If $0<\mu<\frac{\pi}{2}$, then
$\left(\frac{1}{2ed{\pi}^{d-1}}\right)^{\frac{1}{d}}N^{\frac{1}{d}}\lambda<\mu$.
\end{proposition}
\begin{proof}
One can rewrite (\ref{4-S}) using the integral representation of volume of balls in $\Ss^d$ (\cite{CoFr}) as follows:
$$
Nd\omega_d\int_{\frac{\pi}{2}-\frac{\lambda}{2}}^{\frac{\pi}{2}} (\cos t)^{d-1} dt=d\omega_d\int_{\frac{\pi}{2}-\mu}^{\frac{\pi}{2}} (\cos t)^{d-1} dt,
$$
where $\omega_d:=V_{\Ee^d}(\B_{\Ee^d}[\x, 1])$, $\x\in\Ee^d$. Then Lemma 4.7 of \cite{BeLi} yields the following chain of inequalities in a rather straightforward way:
$$
\frac{N}{2ed{\pi}^{d-1}}{\lambda}^d<\frac{N}{ed}\frac{\lambda}{2}\left(\sin{\frac{\lambda}{2}}\right)^{d-1}\leq N\int_{\frac{\pi}{2}-\frac{\lambda}{2}}^{\frac{\pi}{2}} (\cos t)^{d-1} dt=\int_{\frac{\pi}{2}-\mu}^{\frac{\pi}{2}} (\cos t)^{d-1} dt\leq\mu(\sin \mu)^{d-1}\leq{\mu}^d.
$$ From this the claim follows.
\end{proof}
Now Theorem~\ref{main}, (\ref{2-S}), (\ref{3-S}), and (\ref{4-S}) imply in a straightforward way that
\begin{equation}\label{5-S}
V_{\Ss^d}\left(P^{\delta}\right)=V_{\Ss^d}\left(\left(\bigcup_{i=1}^{N}\B_{\Ss^d}\left[\p_i,\frac{\lambda}{2}\right]\right)^{\delta+\frac{\lambda}{2}}\right)\leq
V_{\Ss^d}\left(\left(\B_{\Ss^d}\left[\p_1, \mu\right]\right)^{\delta+\frac{\lambda}{2}}\right)
\end{equation}
Clearly, $\left(\B_{\Ss^d}\left[\p_1, \mu\right]\right)^{\delta+\frac{\lambda}{2}}=\B_{\Ss^d}\left[\p_1,\delta+\frac{\lambda}{2}-\mu\right]$ (with the convention that if $\delta+\frac{\lambda}{2}-\mu<0$, then of course, $\B_{\Ss^d}\left[\p_1,\delta+\frac{\lambda}{2}-\mu\right]=\emptyset$). By assumption $0<\delta+\frac{\lambda}{2}<\frac{\pi}{2}$ and so, if
$\delta+\frac{\lambda}{2}-\mu\geq 0$, then necessarily $0<\mu<\frac{\pi}{2}$. Thus, Proposition~\ref{S-mu} and (\ref{5-S}) yield
\begin{equation}\label{6-S}
g_{\Ss^d}(N,\lambda,\delta)\leq V_{\Ss^d}\left(\B_{\Ss^d}\left[\p_1,\delta-\left(\left(\frac{1}{2ed{\pi}^{d-1}}\right)^{\frac{1}{d}}N^{\frac{1}{d}}-\frac{1}{2}\right)\lambda\right]\right)
\end{equation}
(with the convention that $V_{\Ss^d}(\emptyset)=0$). As $N\geq 2ed{\pi}^{d-1}\left(\frac{1}{2}+\frac{\pi}{2\sqrt{2}}\right)^d$ therefore $\left(\left(\frac{1}{2ed{\pi}^{d-1}}\right)^{\frac{1}{d}}N^{\frac{1}{d}}-\frac{1}{2}\right)\lambda\geq \frac{\pi}{2\sqrt{2}}\lambda$ and so, (\ref{1-S}) and (\ref{6-S}) yield $g_{\Ss^d}(N,\lambda,\delta)<
f_{\Ss^d}(N,\lambda,\delta)$, finishing the proof of (ii) in Theorem~\ref{appl}.

\subsection{Proof of (iii) in Theorem~\ref{appl}}

Let us lower bound (\ref{d-1}) in a way similar to the previous cases. Let $R:={\rm cr}Q$. Then Jung's theorem (\cite{De}) yields $\sinh R\leq \sqrt{\frac{2d}{d+1}}\sinh \frac{\lambda}{2}$. By assumption we have $0<\frac{1}{2}\lambda<\frac{\sinh k}{\sqrt{2}k}\lambda\leq\delta<k$ and so,
\begin{equation}\label{1-H}
0\leq R\leq \sinh R \leq \sqrt{\frac{2d}{d+1}}\sinh \frac{\lambda}{2}<\sqrt{2}\frac{\sinh k}{k}\frac{\lambda}{2},
\end{equation}
where for the last inequality we have used the simple fact that $0<x<\sinh x< \frac{\sinh k}{k}x$ holds for all $0<x<k$. From (\ref{1-H}) it follows that $0\leq R<\frac{\sinh k}{\sqrt{2}k}\lambda$. Thus, $\B_{\Hh^d}\left[\x, \delta-\frac{\sinh k}{\sqrt{2}k}\lambda\right] \subset Q^{\delta}$ for some $\x\in\Hh^d$. (We note that by assumption $\delta-\frac{\sinh k}{\sqrt{2}k}\lambda\geq0$.) As a result we get that
\begin{equation}\label{1-H}
f_{\Hh^d}(N,\lambda,\delta)>V_{\Hh^d}\left(\B_{\Hh^d}\left[\x, \delta-\frac{\sinh k}{\sqrt{2}k}\lambda\right]\right).
\end{equation}
Next,  we upper bound (\ref{d-2}). It follows in a straightforward way that 
\begin{equation}\label{2-H}
P^{\delta}=\left(\bigcup_{i=1}^{N}\B_{\Hh^d}\left[\p_i,\frac{\lambda}{2}\right]\right)^{\delta+\frac{\lambda}{2}}, 
\end{equation}
where the balls $\B_{\Hh^d}[\p_1,\frac{\lambda}{2}],\dots ,\B_{\Hh^d}[\p_N,\frac{\lambda}{2}]$ are pairwise non-overlapping in $\Hh^d$. Thus, 
\begin{equation}\label{3-H}
V_{\Hh^d}\left(\bigcup_{i=1}^{N}\B_{\Hh^d}\left[\p_i,\frac{\lambda}{2}\right]\right)=N V_{\Hh^d}\left( \B_{\Hh^d}\left[\p_1,\frac{\lambda}{2}\right]\right).
\end{equation}
Let $\mu>0$ be chosen such that 
\begin{equation}\label{4-H}
N V_{\Hh^d}\left( \B_{\Hh^d}\left[\p_1,\frac{\lambda}{2}\right]\right)=V_{\Hh^d}\left( \B_{\Hh^d}\left[\p_1,\mu\right]\right).
\end{equation}
Now Theorem~\ref{main}, (\ref{2-H}), (\ref{3-H}), and (\ref{4-H}) imply in a straightforward way that
\begin{equation}\label{5-H}
V_{\Hh^d}\left(P^{\delta}\right)=V_{\Hh^d}\left(\left(\bigcup_{i=1}^{N}\B_{\Hh^d}\left[\p_i,\frac{\lambda}{2}\right]\right)^{\delta+\frac{\lambda}{2}}\right)\leq
V_{\Hh^d}\left(\left(\B_{\Hh^d}\left[\p_1, \mu\right]\right)^{\delta+\frac{\lambda}{2}}\right)=V_{\Hh^d}\left(\B_{\Hh^d}\left[\p_1,\delta+\frac{\lambda}{2}-\mu\right]\right)
\end{equation}
with the convention that if $\delta+\frac{\lambda}{2}-\mu<0$, then $\B_{\Hh^d}\left[\p_1,\delta+\frac{\lambda}{2}-\mu\right]=\emptyset$.
\begin{proposition}\label{H-mu}
If $0<\mu\leq\delta+\frac{\lambda}{2}$, then $\left(\frac{2k}{\sinh 2k}\right)^{\frac{d-1}{d}}N^{\frac{1}{d}}\frac{\lambda}{2}<\mu$.
\end{proposition}
\begin{proof}
One can rewrite (\ref{4-H}) using the integral representation of volume of balls in $\Hh^d$ (\cite{CoFr}) as follows:
$$
Nd\omega_d\int_{0}^{\frac{\lambda}{2}} (\sinh t)^{d-1} dt=d\omega_d\int_{0}^{\mu} (\sinh t)^{d-1} dt.
$$
As $0<\mu\leq\delta+\frac{\lambda}{2}$ therefore by assumption also the inequalities $0<\mu\leq\delta+\frac{\lambda}{2}<2\delta<2k$ hold. Hence the following chain of inequalities follows in a rather straightforward way:
{\small
$$
\frac{N}{d}\left(\frac{\lambda}{2}\right)^d=N\int_{0}^{\frac{\lambda}{2}}t^{d-1}dt< N\int_{0}^{\frac{\lambda}{2}} (\sinh t)^{d-1} dt=\int_{0}^{\mu} (\sinh t)^{d-1} dt < \int_{0}^{\mu}\left(\frac{\sinh 2k}{2k}t\right)^{d-1}dt=\left(\frac{\sinh 2k}{2k}\right)^{d-1}\frac{\mu^d}{d},
$$
}where for the last inequality we have used $0<x<\sinh x<\frac{\sinh 2k}{2k} x$ that holds for all $0<x<2k$. From this the claim follows.
\end{proof}
Thus, Proposition~\ref{H-mu} and (\ref{5-H}) yield
\begin{equation}\label{6-H}
g_{\Hh^d}(N,\lambda,\delta)\leq V_{\Hh^d}\left(\B_{\Hh^d}\left[\p_1,\delta-\left(\left(\frac{2k}{\sinh 2k}\right)^{\frac{d-1}{d}}N^{\frac{1}{d}}-1\right)\frac{\lambda}{2}\right]\right)
\end{equation}
(with the convention that $V_{\Hh^d}(\emptyset)=0$). As $N\geq\left(\frac{\sinh 2k}{2k}\right)^{d-1}\left(\frac{\sqrt{2}\sinh k}{k}+1\right)^d$ therefore $$\left(\left(\frac{2k}{\sinh 2k}\right)^{\frac{d-1}{d}}N^{\frac{1}{d}}-1\right)\frac{\lambda}{2}\geq \frac{\sinh k}{\sqrt{2}k}\lambda$$ and so, (\ref{1-H}) and (\ref{6-H}) yield $g_{\Hh^d}(N,\lambda,\delta)<
f_{\Hh^d}(N,\lambda,\delta)$, finishing the proof of (iii) in Theorem~\ref{appl}.

\bigskip


\noindent K\'aroly Bezdek \\
\small{Department of Mathematics and Statistics, University of Calgary, Canada}\\
\small{Department of Mathematics, University of Pannonia, Veszpr\'em, Hungary\\
\small{E-mail: \texttt{bezdek@math.ucalgary.ca}}

\end{document}